\documentclass{amsart}
\usepackage{epsfig, amsmath, amssymb, latexsym,  amscd}
\usepackage{amsthm}
\newcommand{ \pp }{{\mathbb P}}
\newcommand{ \cc }{{\mathbb C}}
\newcommand{ \II }{{\mathbb I}}
\newcommand{ \ZZ }{{\mathbb Z}}
\newcommand{ \ii }{{\mathcal I}}
\newcommand{ \oo }{{\mathcal O}}

\newtheorem{thm}{Theorem}[section]
 \newtheorem{cor}[thm]{Corollary}
 \newtheorem{lem}[thm]{Lemma}
 \newtheorem{prop}[thm]{Proposition}
\newtheorem{prob}[thm]{Problem} 
 \theoremstyle{definition}
 \newtheorem{defn}[thm]{Definition}
 \theoremstyle{remark}
 \newtheorem{rem}[thm]{Remark}
  \newtheorem{ex}[thm]{Example}

\begin{document}

\title{Determinantal representation  and subschemes of general plane curves}

\author[L. Chiantini]{Luca Chiantini}
\address[L. Chiantini]{Dipartimento di Scienze Matematiche e Informatiche \\ Universit\`{a} di Siena \\ Pian dei Mantellini 44 \\ 53100 Siena,  Italia}
\email{chiantini@unisi.it}
\urladdr{\tt{http://www.mat.unisi.it/newsito/docente.php?id=4}}

\author[J. Migliore]{Juan Migliore}
\address[J. Migliore]{Department of Mathematics\\ University of Notre Dame\\  Notre Dame, IN, 46556, USA}
\email{migliore.1@nd.edu}
\urladdr{\tt{www.nd.edu/~jmiglior}}

\footnote{This paper was written while the second author was sponsored by the 
National Security Agency under Grant Number H98230-09-1-0031.}

\begin{abstract} 
Let $M = (m_{ij})$ be an $n \times n$ square matrix of integers.  For our purposes, we can assume without loss of generality that $M$ is homogeneous and that the entries are non-increasing going leftward and downward.  Let $d$ be the sum of the entries on either diagonal.  We give a complete characterization of which such matrices have the property that a general form of degree $d$ in $\cc[x_0,x_1,x_2]$ can be written as the determinant of a matrix of forms $(f_{ij})$ with $\deg f_{ij} = m_{ij}$ (of course $f_{ij} = 0$ if $m_{ij} < 0$).  As a consequence, we answer the related question of which $(n-1) \times n$ matrices $Q$ of integers have the property that a general plane curve of degree $d$ contains a zero-dimensional subscheme whose degree Hilbert-Burch matrix is $Q$.  This leads to an algorithmic method to determine properties of linear series contained in general plane curves.

\end{abstract}


 \maketitle

\thispagestyle{empty}

\section{Introduction}

The possibility of representing a general homogeneous polynomial of degree $d$
in a polynomial ring $\cc[x_0,\dots,x_r]$ as a determinant of a matrix of polynomials
of lower degree, has been studied in connection with its application to
several theories in Algebra, Analysis and Geometry.
For many applications, indeed, the attention is restricted
to matrices $N$ of linear forms. In this respect, the problem is
essentially well understood (see  e.g. \cite{V} and \cite{beauv}, 
for the state of art of the theory).

The problem, however, makes sense even if we allow $N$ to be a more
general square matrix of forms (= homogeneous polynomials), except that,
as we want the determinant to be homogeneous, some hypothesis on the
degrees of the entries of $N$ is necessary. 

So,  we fix the size and the degrees of the entries of $N$, i.e.
we fix the {\it degree matrix} $M$ of $N$, and assume that $M$ is {\it homogeneous}
(see the definition in the next section).
Our problem is to determine whether or not a general  form of 
degree $d$ in  $\cc[x_0,\dots,x_r]$ can be realized as the determinant
of a matrix of forms $N=(f_{ij})$, with $\deg(f_{ij})=m_{ij}$, for any
given, homogeneous matrix of integers $M=(m_{ij})$ of degree $d$.
 
As it happens for matrices of linear forms, as explained e.g. in \cite{beauv}, 
the answer to the previous question
is negative when $r>2$, except for
$r=3, d\leq 3$. This is essentially a consequence of the 
Noether--Lefschetz principle: a hypersurface $F$ has a determinantal 
 equation if and only if it contains an arithmetically Cohen--Macaulay 
 divisor (that is not a complete intersection on $F$, except for trivial
matrices).
The reasons for this connection follows from the resolution of 
Cohen--Macaulay subschemes of codimension 2 in projective spaces,
and the Hilbert-Burch theorem. Details can be read in \cite{CGO} or \cite{E1}.

So, we will restrict ourselves to $r=2$, i.e. to the case
of plane curves. 

In this situation,  it is classically known (see e.g. \cite{Dickson}, 
or \cite{V} for a modern account of the theory)
 that any form of degree
$d$ is the determinant of a $d\times d$ matrix of linear forms.
Representations of equations of plane curves as determinants
of other types are present in the literature.
Beauville essentially proves (see \cite{beauv}, Proposition 3.5)
that when the entries of $M$ all belong to $\{1,2\}$, 
then the representation is possible, for a general form. 
Another particular case, related with Petri's theory of special 
linear series, is treated in \cite{AC}. 

We want to complete the picture, and answer the following problem:

\begin{prob}\label{first} Given a general homogeneous form $F$ of degree $d$
in $\cc[x_0,x_1,x_2]$ (representing a general plane curve), and a 
homogeneous matrix of integers, $M$,
of degree $d$, can we find a matrix $N$ of forms, whose degree matrix
is $M$, and such that $det(N)=F$?
\end{prob} 

The solution of the previous problem will also give a criterion for
deciding which 0-dimensional subschemes (i.e. sets of points)
one finds on a general plane curve of degree $d$. 

A solution of the $2\times 2$ case comes out,  as a by-product,
from the main result contained in \cite{CaChGe}. It turns out
that, for any choice of a homogeneous, degree $d$,
 $2\times 2$ matrix $M$ of {\it non-negative} integers, the general 
 form of degree $d$ in $\cc[x_0,x_1,x_2]$ is the determinant
 of a matrix of forms whose degree matrix is $M$.
 
Using the result of \cite{CaChGe} as the initial step of our induction,
we will show that the answer to the problem is positive  
if and only if $M$ satisfies mild natural combinatorial conditions
explained below (see Theorem \ref{main}).

Namely, as we point out in Examples \ref{exdiag} and \ref{exsubdiag}
below, it is easy to see that if an ordered matrix $M$ has negative entries in the 
main diagonal, then a general form cannot be the determinant of a matrix 
whose degree matrix is $M$.
The same happens when $M$ has negative entries in the subdiagonal,
except for trivial cases.

We will see, in the main theorem, that if one excludes the two previous,
obviously negative cases,
then a general form of degree $d$ is always the determinant of a matrix,
whose degree matrix is a pre-assigned $M$.

As a consequence of our result, we give a procedure to determine whether
or not sets of points with a prescribed degree Hilbert-Burch matrix (or 
with a prescribed Hilbert function) are contained in a general curve of given degree.
Via the adjunction process, this last result can be used to 
determine properties of linear series contained on general plane curves,
as explained in the last section of the paper. 

The paper is organized as follows. In section 2, we give our starting definition
of homogeneous matrices and ordering. In section 3, we point out the relation
between determinantal representations of plane curves
and $0$-dimensional subschemes. In section 4, we prove a lemma
on families of $0$-dimensional schemes, which provides the main tool
for our induction. Section 5 is devoted to the proof of the main theorem. 
In section 6, we show how the main result can be used to detect the
existence, on general plane curves, of divisors with prescribed invariants
and linear series with prescribed properties. The authors wish to thank 
the referee for pointing out an incorrect statement in a previous 
version of this section.

\section{Foundations}

\begin{defn} 
A $2\times 2 $ matrix of integers
$$ M = \begin{pmatrix} a& b \\ c & e\end{pmatrix}$$
is {\it homogeneous} if $a+e = b+c$.\par 
A  matrix $M=(m_{ij})$ of integers is 
$\it homogeneous$ if all its $2\times 2$ submatrices are.\par 
It turns out that a square $n\times n$ matrix $M$ is homogeneous if,
for any permutation $\sigma$ in the set $\{1,\dots,n\}$, the number
$$  d = m_{1\sigma(1)}+ m_{2\sigma(2)}+\dots + m_{n\sigma(n)}$$
is constant. Indeed, any permutation can be obtained as a product of a series
of transpositions.  
The number $d$ is also called  the {\it degree} of the matrix.
\end{defn}

It is clear that $M$ is homogeneous if and only if all its 
submatrices are.

\begin{rem}
Let $M=(m_{ij})$ be a $n\times n$  matrix of integers. For any $i,j$
consider a homogeneous form $f_{ij}$ of degree $m_{ij}$ in
the polynomial ring $\cc[x_0,\dots,x_r]$. 
Of course, if $m_{ij}$ is negative then the form $f_{ij}$
must necessarily be $0$. 

If $M$ is homogeneous,
of degree $d$, then the determinant of the matrix $(f_{ij})$ 
is a homogeneous form of degree $d$.\par 
\end{rem}

For a given matrix of integers, let us fix 
a standard ordering. Notice indeed that any permutation of
rows and columns only changes the sign of the determinant,
thus is irrelevant for our problem.

\begin{defn} We say that a matrix of integers $M=(m_{ij})$
is {\it well-ordered} if:
$$ \mbox{ for } i'>i \mbox{ and } j'>j, \mbox{ we have } m_{i'j}\leq m_{ij}
\mbox{ and } m_{ij'}\geq m_{ij}. $$
Roughly speaking, the matrix is non-increasing going leftward
and downward. The maximal element is $m_{1n}$ while the minimal
is $m_{n1}$.
\end{defn}

The following examples point out two natural conditions on the matrix $M$, which
exclude that a general form of degree $d$ is the determinant of a matrix whose degree
matrix is $M$.

\begin{ex}\label{exdiag} Let $M=(m_{ij})$ be a well-ordered, $n\times n$ homogeneous matrix
 of integers of degree $d$. Assume that for  some $k=1,\dots n$, the element 
 $m_{kk}$, in the main diagonal, is negative. Then a general form of degree $d$ in $\cc[x_0,x_1,x_2]$
 is not the determinant of a matrix of forms $N$, whose degree matrix is $M$.
 
 Indeed, the ordering implies that $m_{ij}< 0$ for $i\geq k$ and $j\leq k$. 
 Thus, in the matrix of forms $N=(f_{ij})$, we have $f_{ij}=0$ when
  $i\geq k$ and $j\leq k$. Then the determinant of $N$ is $0$. 
 \end{ex}

 \begin{ex}\label{exsubdiag} Let $M=(m_{ij})$ be a well-ordered, 
 $n\times n$ homogeneous matrix
 of integers of degree $d$. The elements of type $m_{k\ k-1}$,  $k=2,\dots n$, 
 form the so-called {\it sub-diagonal}.
 
 Assume that for  some $k=2,\dots n$, the element 
 $m_{k\ k-1}$ is negative. Then a general form of degree $d$ in $\cc[x_0,x_1,x_2]$
 is not the determinant of a matrix of forms $N$, whose degree matrix is $M$,
 unless the submatrix $M'$ of $M$ obtained by erasing the first $k-1$ rows and columns, 
 has either degree $0$ or degree $d$.
 
 The reason is clear. If for some $k$ we have $m_{k\ k-1}<0$, then necessarily,
 in the matrix of forms $N=(f_{ij})$, we have $f_{ij}=0$ 
 when $i\geq k$ and $j\leq k-1$. Call $N'$ the $(k-1)\times (k-1)$) submatrix of $N$ 
 formed by the first $k-1$ rows and columns and call $N''$ 
 the $(n-k+1)\times(n-k+1)$ matrix obtained from $N$ by erasing these first $k-1$ 
 rows and columns. We get $\det(N)=\det(N')\det(N'')$ and
 $\deg(\det(N''))=d-\deg(\det(N'))$. If $d>\deg(\det(N'))=\deg(M')>0$, 
 the conclusion follows, since the general form of degree $d$ is irreducible.
 \end{ex}

\section{Determinants and subschemes}\label{detsub}

The existence of a determinantal representation for a 
form $F\in \cc[x_0,x_1,x_2]$ is strictly connected with the existence of some sets
of points on the curve $C\subset\pp^2$, associated with $F$.

The reasons for this connection follow from the resolution of zero-dimensional
schemes on the plane, and the Hilbert-Burch theorem. 

Our main idea is to prove the existence of the mentioned sets of points, 
on a general plane curve, using the liaison process and its effects
on the Hilbert-Burch matrix, a classical method introduced years ago by F. Gaeta
(see \cite{G}).

We briefly outline in this section the main features of the connection between
matricial representations and subsets; see \cite{E2} pp.\ 501--503 for details.  
We begin in the more general setting of codimension two arithmetically 
Cohen-Macaulay subschemes of $\mathbb P^r$, and then explain why we restrict to 
zero-dimensional subschemes of $\mathbb P^2$.

\smallskip

Let $Z\subset \pp^r$ be an arithmetically Cohen-Macaulay scheme
of codimension $2$. Call $\oo$ the structure sheaf of $\pp^r$. The ideal sheaf 
$\ii_Z$ has a free resolution of type:
$$ 
0\to \oplus^{n-1} \oo(-b_j)\stackrel{A}\longrightarrow\oplus^n \oo(-a_i)\to \ii_Z\to 0,
$$
where $A$ is given by a $(n-1)\times n$ matrix of forms, the {\it Hilbert-Burch}
matrix of $Z$. The maximal minors of $A$ are forms of degrees $a_1,\dots, a_n$,
which  generate the homogeneous ideal of $Z$.

If $A=(f_{ij})$ and $\deg(f_{ij})=m_{ij}$, the matrix $D=(m_{ij})$ is thus a
homogeneous matrix of integers, whose minors have degrees $a_1,\dots, a_n$.
It is called a {\it degree Hilbert-Burch} matrix (dHB for short) of $Z$.

In this picture, arranging the numbers so that $a_1\geq\dots\geq a_n$
and $b_1\geq\dots\geq b_{n-1}$, one has $m_{ij}=b_i-a_j$ and the matrix
$D$ is well-ordered.

Notice that we do {\it not} assume that the resolution is minimal. 
Hence the matrix $A$ may have some entries which are non-zero constants
or, equivalently, the map described by $A$ can induce an isomorphism
on some factors $\oo(-b_j)\to \oo (-a_i)$, with $b_j=a_i$, i.e. $m_{ij}=0$.

In this sense, the numbers $a_i,b_j$ are not uniquely determined by $Z$,
since it is always possible to add redundant factors in the resolution.
\smallskip

Conversely, let $D=(m_{ij})$ be a $(n-1)\times n$ well-ordered,
homogeneous matrix of integers and let $A$ be a matrix of forms,
in $\cc[x_0,\dots, x_r]$, whose degree matrix is $D$.

Let $a_j$ be the degree of the minor obtained by erasing the $i$-th
column and take $b_i=a_j+m_{ij}$. 

The matrix $A$ determines a map of sheaves 
$$ \oplus^{n-1} \oo(-b_j)\stackrel{A}\longrightarrow\oplus^n \oo(-a_i).$$
When the map injects, and drops rank in codimension $2$, then the
cokernel is the ideal sheaf of an arithmetically Cohen-Macaulay subscheme
of codimension $2$, whose homogeneous ideal is generated by the maximal 
minors of $A$.

\begin{rem} Assume that, for some $k$, $m_{kk}<0$. Then the ordering implies that
$m_{ij}<0$ for $i\geq k$ and $j\leq k$. Thus, in the matrix $A$, we have 
$f_{ij}=0$ for $i\geq k$ and $j\leq k$. Hence the maximal minors of $A$ are
either $0$, or they contain the common factor $\det(A')$, where $A'$
is the square matrix obtained by deleting the first $(k-1)$ rows 
and the first $k$ columns of $A$.

It follows that the map defined by $A$ drops rank in the locus defined by
$\det(A')=0$, which has codimension at most $1$, unless $A'$ has degree $0$.
\end{rem}

\begin{rem} Assume $m_{kk}=0$ for all $k=1,\dots, n$. Then $a_n=0$.
Thus, for a general choice of the forms $f_{ij}$, the determinant
of the last minor of $A$ is a non-zero constant, and thus the map
drops rank nowhere.

From some point of view, this can be considered as a degenerate case of the
general situation, in which the locus where the matrix $A$ drops rank has degree $0$
and is empty (so its {\it true} dimension is $-1$). 
\end{rem}

Excluding the previous two cases, i.e. when $m_{kk}\geq 0$ for all $k$ and
$\max\{m_{kk}\} >0$, then for a general choice of the forms $f_{ij}$
of degrees $m_{ij}$, the resulting matrix $A$ determines a map which is injective,
and drops rank in codimension $2$. Thus $A$ is the Hilbert-Burch matrix
of an arithmetically Cohen-Macaulay subscheme of codimension $2$
(see \cite{CGO}, Remark at the end of 0.4, or see \cite{G}).

The construction yields some consequences:

\begin{prop}\label{matrixpoints} Let $M$ be a $n\times n$ homogeneous matrix of integers.

If a general hypersurface of degree $d$ in $\pp^r$ contains an
arithmetically Cohen-Macaulay scheme $Z$ of codimension $2$, with a dHB
matrix that corresponds to the submatrix of $M$ obtained by erasing one 
row of $M$, then the a general
form of degree $d$ in $\cc[x_0,\dots, x_r]$ is the determinant
of a matrix of forms, whose degree matrix is $M$.

Conversely, assume that a general form of degree $d$ is the determinant
of a matrix of forms, $N$, whose degree matrix is $M$. Let $R=(r_{ij})$ be a matrix
obtained by erasing one row of $M$. Assume $r_{kk}\geq 0$ for all
$k=1,\dots, n-1$ and assume $\max\{r_{kk}\}>0$. Then the hypersurface
defined by $N$ contains an arithmetically Cohen-Macaulay subscheme
of codimension $2$, with a dHB matrix equal to $R$.
\end{prop}

\begin{rem} If we drop the assumption `$r_{kk}\geq 0$ for all
$k=1,\dots, n-1$', the converse of the previous statement may fail.
E.g. consider the matrix:
$$ M= \begin{pmatrix}
0 &1 &10 &11 \\ -1 &0 &9 &10 \\ -5 &-4 &5 &6 \\ -8 &-7& 2& 3
\end{pmatrix}. $$

It follows from our main theorem below, that a general form of degree $8$
in three variables is the determinant of a matrix 
whose degree matrix is $M$. Nevertheless,
by erasing the first row, we do not find the dHB matrix
of a codimension two subscheme.

This shows that it can sometimes happen that $r_{kk} < 0$.  
Indeed, if $R$ is obtained from $M$ by erasing the $q$-th row, 
then $r_{kk} = m_{kk}$ for $k < q$, while $r_{kk} = m_{k+1 \ k}$ 
for $k \geq q$.  Thus, if a general form of degree $d$ is the 
determinant of a matrix of forms whose degree matrix is $M$, 
then $r_{kk} < 0$ can only happen for $k \geq q$ (by Example \ref{exdiag}).  
Moreover, in this case, $M$ must have the shape described in 
Example \ref{exsubdiag}, which we will consider apart, 
in Remark \ref{rem1}. 
 
Notice also that if we drop the assumption $\max\{r_{kk}\}>0$
in the previous proposition, then the argument also works,
with the only exception that the subscheme $Z$ could be empty!
\end{rem}

As in $\pp^2$ every subscheme of codimension $2$
is arithmetically Cohen-Macaulay, a problem very similar to the one 
stated in the previous section is the following:

 \begin{prob}\label{second} 
Does a general curve of degree $d$ contain a zero-dimensional subset $Z$,
whose dHB degree matrix is a preassigned $(n-1)\times n$ 
homogeneous matrix $D$ of integers?
\end{prob} 

The starting point of our analysis is the following theorem,
which settles the case in which $M$ is a $2\times 2$ matrix.
In this case, $Z$ has a homogeneous ideal generated by
$2$ forms, i.e. it is a complete intersection. 

\begin{thm}\label{2x2} Let $M=(m_{ij})$ be a homogeneous, ordered, $2\times 2$ matrix
of integers. Call $d$ the degree of $M$.

A general form $F$ of degree $d>0$ in $\cc[x_0,x_1,x_2]$ is the determinant
of a matrix of forms whose degree matrix is $M$, if and only if either:\par 
- $m_{11}=0,d $; or \par
- $m_{21}\geq 0$. 
\end{thm}
\begin{proof} The conditions are necessary, as explained in Examples \ref{exdiag} and
\ref{exsubdiag}. 

For the converse, notice that the first case is trivial.
Thus, assume $m_{11}\neq 0,d$ and $m_{12}\geq 0$. 
Notice that $m_{11},m_{22}\geq m_{21}$,
hence $d\geq m_{11}>0$. Also $m_{12}=d-m_{21}$ hence $d\geq m_{12}\geq m_{11}>0$.

The main theorem of \cite{CaChGe} shows that a general curve of degree $d$
contains the complete intersection $Z$ of curves of degree $m_{11}, m_{12}$.
Since a dHB matrix of $Z$ is $(m_{11}\ m_{12})$, the claim follows by
Proposition \ref{matrixpoints}.
\end{proof}

We end this section by stressing an important, although trivial, remark.

Starting with a matrix of forms $N$, and erasing
{\it different} rows, we get a priori different $(n-1)\times n$ degree matrices.
Thus, there are a priori zero-dimensional schemes whose resolutions have
rather different numerical invariants, whose existence on a general
 curve of degree $d$ implies that a general form is the determinant
 of a matrix, with degree matrix $M$.
 
We will use this observation several times, when constructing
our inductive argument about the representation of forms as determinants.

\section{An incidence variety}\label{inc}

Let $T$ be an irreducible subvariety of the Hilbert scheme of points
in $\pp^2$, such that the dHB matrix is constant along $T$.
Then also the Hilbert function $Hf$ is constant along $T$.
Call $\delta$ the degree of elements in $T$.

Consider the incidence variety:
$$\II =\II(d) :=\{(C,Z): C \mbox{ is a curve of degree $d$ containing }Z\in T \}$$
with the two projections $p=p(d):\II\to T$ and $q=q(d):\II\to \pp(H^0O(d))$.
We want to study conditions under which $q(d)$ dominates $\pp(H^0\oo(d))$, which
amounts to saying that a general curve of degree $d$ contains a set of points
in $T$.

For the application to our problem on
the determinantal representation of plane curves, it would
be sufficient to consider the case in which $T$ is the whole stratum
of the Hilbert scheme of points with fixed dHB matrix
(which is irreducible, since it is dominated
by a product of projective spaces).
Since the result we are going to use is indeed general,
we will  maintain the generality of $T$, throughout this section.

By construction, the fiber of $p(d)$ at $Z\in T$ is $\pp(H^0\ii_Z(d))$,
$\ii_Z$ being the ideal sheaf of $Z$. 
It follows that $\II$ is irreducible. Moreover:
$$\dim \II(d) = \dim(T)+h^0\ii_Z(d)-1 = \dim(T)+h^0\oo(d)-Hf(d)-1.$$
Since $\pp(H^0\oo(d))$ has dimension $h^0\oo(d)-1$, we get immediately:

\begin{prop} If $\dim(T)< Hf(d)$, then $q(d)$ cannot be dominant.
\end{prop}

The fundamental remark is the following result (see e.g. \cite{ChFa}, Lemma 3.2).

\begin{thm}\label{split}
Assume that, for $Z\in T$, the Hilbert function $Hf(d)$ coincides
with the degree $\delta$. Assume that $q(d)$ is dominant.
Then for all $d'\geq d$, also $q(d')$ is dominant.
\end{thm}
\begin{proof}
For the sake of completeness, we sketch here the argument, 
that is contained in the proof of \cite{ChFa}, Lemma 3.2.

It is clearly sufficient to prove the theorem for $d'=d+1$.

Notice that any component $F_{d+1}$ of a general fiber  of $q(d+1)$ has dimension:
$$\dim F_{d+1}\geq  \dim\II(d+1) -\dim(\pp(H^0O(d+1))) = \dim(T)-Hf(d+1)=\dim(T)-\delta$$
and the inequality is strict, if  $q(d+1)$ does not dominate. 

Since, by assumption, $q(d)$ dominates, the same computation shows
that the dimension of a general fiber
of $q(d)$ is $\dim(T)-\delta$. We simply want to know that the dimension
does not change passing from $d$ to $d+1$.

Let $(C,Z)$ be a general point in $\II(d)$. Then, the fiber of $q(d)$
over $C$ has dimension $\dim(T)-\delta$, in a neighbourhood of $(C,Z)$.
Consider now a general line $L$ and put $C'=C\cup L$. The pair
$(C',Z)$ sits in $\II(d+1)$ and, in particular, it sits in the
fiber of $q(d+1)$ over $C'$. Moreover, in a neighbourhood 
of $(C',Z)$, all pairs in the fiber of $q(d+1)$ over $C'$ are
of type $(C',Z')$ with $Z'\subset C$. It follows that the fiber
of $q(d+1)$ over $C'$ has at least one component of dimension
$\dim(T)-\delta$. Since by \cite{AG}, 3.22.b page 95, the dimension
of components of fibers can only increase under specialization,
it follows that a general fiber of $q(d+1)$ has at least
one component of dimension $\dim(T)-\delta$. This proves that
$q(d+1)$ dominates.
\end{proof}

The previous result, which indeed has validity far beyond
our application to points in $\pp^2$, implies that 
in order to show that general curves of any high degree $d'$
contain a subscheme with fixed dHB matrix, it is sufficient to prove
the claim for some $d\leq d'$ for which the Hilbert function
achieves the degree.

\begin{rem} \label{getdeg}
Consider a zero-dimensional subscheme $Z\subset\pp^2$, whose
ideal sheaf has a free resolution:
$$ 0\to \oplus^{n-1} \oo(-b_j)\stackrel{A}\longrightarrow\oplus^n \oo(-a_i)\to \ii_Z\to 0$$
where, as usual, we put $b_1\geq \dots \geq b_{n-1}$ and $a_1\geq\dots\geq a_n$.
Then for all $d\geq b_1-2$, the Hilbert function of $Z$ at level $d$ coincides
with $\deg(Z)$.
See \cite{CGO}, \S 0, (3), for a proof of this fact.
 \end{rem}

\section{The main result}

Now we are ready to state and prove our main result on
the determinantal representation of plane curves.
\smallskip

\begin{thm}\label{main} {\bf (main)} Let $M$ be a homogeneous $n\times n$  
matrix of integers, of degree $d$. Assume $M$ is well-ordered.

Then a general form of degree $d$ in $\cc[x,y,z]$ is the determinant
of a matrix of forms $N$, whose degree matrix is $M$, if and only if
the two conditions hold:

1) $m_{ii}\geq 0$ for all $i$;

2) Whenever for  some $k=2,\dots n$, the element 
 $m_{k\ k-1}$ is negative, then the submatrix $M'$ of $M$ obtained by 
 erasing the first $k-1$ rows and columns,  has degree $0$ or $d$.
\end{thm}

We will make induction on the size $n$ of $M$.

With a series of remarks, we reduce ourselves to proving
only the existence of the matrix $N$, when $m_{k\ k-1}\geq 0$ for all $k=2,\dots,n$.

\begin{rem} Conditions 1) and 2) of the theorem are 
necessary, as explained in Examples \ref{exdiag} and \ref{exsubdiag}.
\end{rem}

\begin{rem} The theorem is trivial, when $n=1$. When $n=2$, the theorem
is an easy consequence of  \cite{CaChGe}, as explained in
Theorem \ref{2x2}. Namely, notice that the conditions 1) and 2) of the main theorem
reduce to conditions of Theorem \ref{2x2}, when $n=2$.
\end{rem}

Let us see what happens when  $m_{k\ k-1}<0$ and conditions 1) and 2) of the 
main theorem hold.

\begin{rem}\label{rem1} With the notation of the theorem, assume
$m_{k\ k-1}<0$. Then also $m_{ij}<0$ for $i\geq k$ and $j\leq k-1$.
Hence the matrix of forms $N=(n_{ij})$ necessarily has $n_{ij}=0$
for $i\geq k$ and $j\leq k-1$. Thus if $N'$ is the matrix
formed by the first $k-1$ rows and columns of $N$, and $N''$ is obtained
from $N$ by erasing these rows and columns, then $det(N)=\det(N')\cdot\det(N'')$.

Call $e$ the degree of $M'$, which is the degree matrix of $N'$.
The degree matrix of $N''$ is homogeneous, of degree $d-e$.
By condition 2) of the main theorem, we must have either $e=d$ or $e=0$.
By induction, when $e=d$, for a general choice of $N$, a general  form of
degree $d$ is the determinant of $N'$, while a general constant
is the determinant of $N''$. Thus the theorem holds, in this case.
The case $e=0$ is similar.
\end{rem}

Next, let us see what happens when $m_{kk}=0$ for some $k$. 

\begin{rem}\label{rem2} With the assumptions of the main theorem, assume
$m_{kk}=0$ for some $k$. Then, erasing the $k$-th row and column,
we get a $(n-1)\times(n-1)$ matrix $M'$ which is again 
homogeneous and satisfies the condition of the theorem. Thus, by
induction, a general  form $F$ of degree $d$ is the determinant of a matrix 
$N'$ whose degree matrix is $M'$. Adding to $N'$, in the $k$-th position, 
a row and a column which are zero, except for $n_{kk}=1$, we get a new square
$n\times n$ matrix $N$, whose degree matrix is $M$ and whose
determinant is $F$.
\end{rem}

It follows from the previous remarks, that the theorem is proved
once one shows that for any $n\times n$ 
homogeneous well-ordered matrix $M$ of degree
$d$, with $m_{kk}>0$, $m_{k\ k-1}\geq 0$, for all $k$,
then a general form of degree $d$ in
$\cc[x,y,z]$ is the determinant
of a matrix of forms, whose degree matrix is $M$.
\smallskip

\begin{lem}\label{indu} Assume, by induction, that the main theorem
holds for matrices of size $(n-1)\times(n-1)$.
Fix a homogeneous well-ordered matrix $Q=(q_{ij})$ of size $(n-2)\times(n-1)$ and
call $a_1\geq a_2\geq \dots\geq a_{n-1}$ the degrees of its maximal minors.
Assume $q_{kk}>0$ for all $k$, and fix an integer $d\geq a_1$. 

Then a general form of degree $d$ corresponds to a curve $C$ which contains
a set of points with a dHB matrix equal to $Q$.
\end{lem}
\begin{proof}
Add to $Q$ the row $(d-a_1,\dots , d-a_{n-1})$ and reorder.
We get a $(n-1)\times(n-1)$ homogeneous matrix $M=(m_{ij})$.
We have the following possibilities for $M$.

\begin{enumerate}
\item if $d-a_k\geq q_{kk}$ then $m_{k+1\ k}=q_{kk}>0$.

\item if $d-a_k\leq q_{k+1\ k}$ then $m_{k+1\ k}=q_{k+1\ k}$.
But since $d-a_k>0$, we also get $m_{k+1\ k}>0$.

\item if $q_{k+1\ k}<d-a_k< q_{kk}$ then $m_{k+1\ k}=d-a_k>0$.
\end{enumerate} 
 Then by induction, a general form of degree $d$ is the determinant
 of a matrix of forms, whose degree matrix is $M$. The claim follows. 
\end{proof}

\begin{lem}\label{zero}
Assume the main theorem holds for matrices of size $(n-1)\times(n-1)$.
Take a homogeneous $n\times n$ matrix $M$ as in the theorem and
assume $m_{i1}=0$ for some $i$.
 Then a general form of degree
$d$ is the determinant of a matrix of forms, whose degree matrix is $M$. 
\end{lem}
\begin{proof}
Let $Q=(q_{ij})$ be the matrix obtained by erasing the $i$-th row and the first
column of $M$. $Q$ is a homogeneous $(n-1)\times(n-1)$ matrix of degree $d$.
Either $q_{kk}=m_{k\ k+1}$, for $k<i$, or $q_{kk}=m_{k+1\ k+1}$.
In any event $q_{kk}\geq 0$. If $k<i-1$ then $q_{k+1\ k}=m_{k+1\ k+1}\geq 0$. 
If $k\geq i-1$ then $q_{k+1\ k}=m_{k+2\ k+1}$, and the matrix obtained by
erasing the first $k$ rows and columns of $Q$ coincides with the matrix
obtained from $M$ by erasing the first $k+1$ rows and columns.
Thus $Q$ satisfies the assumptions of the main theorem, and by induction 
we know that a general form $F$ of degree $d$ is the determinant of a matrix of
forms $N$, whose degree matrix is $Q$.

Now, adding to $N$ a $i$-th row of type $(1\ 0\ 0\ \dots 0)$ and a 
first column of type $(0\ 0 \ \dots 0\ 1\ 0 \dots 0)$, with $1$ in the $i$-th place,
we get a matrix of forms, whose determinant is $F$ and whose 
degree matrix is $M$.
\end{proof}

\begin{prop}\label{inductive} 
Let $M$ be a  $n\times n$ homogeneous, well ordered matrix of integers,
of degree $d$.
Let $M'$ be the matrix obtained from $M$
by erasing the first row. Assume there exists a $0$-dimensional set $Z$  with a dHB 
matrix equal to  $M'$. Then the Hilbert function of $Z$ in degree $d-1$
coincides with the degree of $Z$. 
\end{prop}
\begin{proof} It is enough to consider the resolution of the ideal sheaf
$\ii$ of $Z$. We get:
$$
 0\to \oplus ^{n-1} O(-b_i) \to \oplus^n O(-a_i) \to \ii\to 0,
$$
where the $a_i$'s are the degrees of the maximal minors of $M'$.
After ordering the $a_i$'s and the $b_j$'s so that 
$a_1\geq \dots \geq a_n$ and $b_1\geq \dots \geq b_{n-1}$, 
we simply need to prove that $d>b_1-2$. But we have:
\[
\begin{array}{rcl}
d & =  & m_{11}+ \deg 
\begin{pmatrix} m_{22} &\dots & m_{2n} \\
\vdots & & \vdots \\ m_{n2}& \dots & m_{nn}\end{pmatrix}
 = m_{11}+a_1 \\
b_1 & = & m_{21} + a_1,
\end{array}
\]

and since $m_{11}\geq m_{21}$, the conclusion follows.
\end{proof}

Now we are ready for the proof of the main theorem.
\smallskip

{\bf proof of the main theorem}
If $m_{i1}=0$ for some $i$, then we are done, by
Lemma \ref{zero}. So, we just need to show, by induction, that
we can always reduce to this case.

Recall that $m=m_{11}>0$ is the maximum of the first column;
call $x$ the number of entries in the first column, for which the 
maximum is attained. 

We need to prove that a general curve of degree $d$ contains
a subscheme $Z$ whose dHB matrix is the matrix $M'$ obtained
by erasing the first row of $M$. By Proposition \ref{inductive},
together with Theorem \ref{split}, applied to the stratum $T$
of sets of points with fixed dHB matrix,
 it is enough to prove that a general
curve of degree $d-1$ contains a scheme like $Z$.
But this amounts to saying that a general form of degree $d-1$
is the determinant of a matrix whose degree matrix is 
$$ \bar M = \begin{pmatrix} m_{11}-1 & m_{12}-1 &\dots & \dots & m_{1n}-1 \\
 & & M' & & \end{pmatrix}
$$ 
Notice that, although $\bar M$ is possibly unordered, either
the maximum of the first column of $\bar M$
is smaller than $m$, or the number of entries for which the maximum
is attained is smaller than $x$.

Then, we reorder $\bar M$ (just reordering the rows is enough),
and repeat the procedure. It is clear that, after a finite number of steps,
we end up with a matrix having a zero in the first column, to
which we may apply Lemma \ref{zero}.

The claim follows.\qed

\section{Subschemes and linear systems on a general plane curve}\label{pts}

In this section, we discuss an application of the previous result. Namely,
using the connection between determinantal representation of 
forms and $0$-dimensional subschemes of general curves,
outlined in section \ref{detsub}, we see that we are able to classify
{\it all} the dHB matrices of subsets of points that one
can find on  plane curves of degree $d$.
\smallskip

The procedure goes as follows.

\begin{prob} Fix a well-ordered 
dHB matrix $Q=(q_{ij})$, i.e. a homogeneous $(n-1)\times n$
matrix of integers, such that $q_{ii}\geq 0$ for all $i$
and $\max\{q_{ii}\}>0$. Fix a degree $d$.

Does a general plane curve of degree $d$ contain a
$0$-dimensional subset, with a dHB matrix equal to $Q$?  
\end{prob}

{\bf Solution.}
Let $a_1,\dots, a_n$ be the degrees of the maximal minors
of $Q$. Consider the row $(d-a_1,\dots, d-a_n)$.
Add the row to $Q$ and reorder, so that the resulting
square matrix $M=(m_{ij})$ is well-ordered.

The answer to the problem is positive if and only if
$M$ satisfies conditions 1) and 2) of the main theorem,
namely $m_{ii}\geq 0$ for all $i$ and
$m_{i\ i-1}<0$ implies that the matrix obtained by erasing 
the first $i-1$ rows and columns of $M$ has either degree $0$
or degree $d$.
\medskip

Let us try to give a direct  characterization
of which schemes one finds on a general plane curve
of degree $d$.

Write, as usual, the  resolution
of the ideal sheaf $\ii$ of a scheme $Z$ as:
$$
0\to \oplus^{n-1}\oo(-b_j)\to \oplus^n\oo(-a_i)\to\ii\to 0
$$
with $b_1\geq\dots\geq b_{n-1}$ and $a_1\geq\dots\geq a_n$
and consider the dHB matrix $Q=(q_{ij})$, $q_{ij}=b_i-a_j$.
We must have $q_{ii}\geq 0$ for all $i$.

For simplicity, we will assume, as we can always do, 
that the resolution is {\it minimal}.
This implies that $b_1>a_i$ and $b_{n-1}>a_n$, i.e.
$q_{11}>0$ and $q_{n-1\ n}>0$ (see section 0 of \cite{CGO}).

\begin{cor} \label{cor of main}
With the previous notation, one has:

\begin{itemize}
\item[(i)] If $d\geq b_1$, then a general curve of degree $d$
contains a scheme with dHB matrix equal to $Q$.

\item[(ii)] Assume $d< b_{n-1}$. Then a general curve of degree $d$
contains a scheme with dHB matrix equal to $Q$ if and only if
either $d=a_n$ or $d\geq a_{n-1}$.

\item[(iii)] Assume there is  $i$ such that $b_{i-1}>d\geq b_i$.
 Then a general curve of degree $d$
contains a scheme with dHB matrix equal to $Q$ if and only if
$q_{k\ k-1}\geq 0$ for $k=2,\dots,i-1$ and $d\geq a_{i-1}$.

\end{itemize}
\end{cor}
\begin{proof}
In case (i), the new row $(d-a_1\dots d-a_n)$ goes at the top, 
in the reordering, so that in the subdiagonal of 
the new square matrix $M=(m_{ij})$ one
has  the elements of the diagonal of $Q$,
which are non-negative.

In case (ii), the new row $(d-a_1\dots d-a_n)$ goes at the bottom, 
so that in the subdiagonal of the new square matrix $M=(m_{ij})$, one
has the elements of the subdiagonal of $Q$,
which are non-negative, and $d-a_{n-1}$. If this last is non-negative,
we are done. Otherwise condition 2) of the main theorem requires
that $m_{nn}=d-a_n$ is $0$.

In case (iii), the new row $(d-a_1\dots d-a_n)$ goes in the $i$-th
position. In the square matrix $M=(m_{ij})$, the elements
of the subdiagonal are $q_{21},\dots, q_{i-1\  i-2}, d-a_{i-1},
q_{i i},\dots, q_{n-1\ n-1}$. Since $q_{ii},\dots, q_{n-1\ n-1}$
are non-negative, we focus our attention on the other elements 
of the subdiagonal.

If $d-a_{i-1}<0$, condition 2) of the
main theorem implies that either $q_{11}+\dots+q_{i-1\ i-1}=0$,
which is impossible since they are non-negative and $q_{11}>0$,
or $(d-a_i)+q_{i\ i+1}+\dots+ q_{n-1\ n}=0$.
As $d-a_i\geq b_i-a_i\geq 0$, and $q_{i\ i+1},\dots, q_{n-1\ n}\geq 0$,
 this last equality implies that $q_{n-1\ n}=0$, which is impossible
 when the resolution is minimal.
 
Similarly, if  $q_{k\ k-1}<0$ for some $k=1\dots i-1$,   
condition 2) of the main theorem implies that either 
$q_{11}+\dots+q_{k-1\ k-1}=0$,
which is impossible since they are non-negative and $q_{11}>0$,
or $q_{kk}+ \dots + q_{i-1\ i-1}+(d-a_i)+q_{i\ i+1}+\dots+ q_{n-1\ n}=0$.
As $d-a_i\geq b_i-a_i\geq 0$, and the other summands are non-negative,
 this last equality implies that $q_{n-1\ n}=0$, impossible
 when the resolution is minimal.
 \end{proof}
 
The previous result yields  the following
``asymptotic" principle:

 \begin{cor}
 For any choice of a (possible)
dHB matrix $M$ of a set of points, and for $d\gg 0$, a general
curve of degree $d$ contains  subschemes whose dHB matrix is $M$. 
 
 \end{cor}

Let see in some examples how  it works.

\begin{ex}
Consider the following homogeneous matrix
$$ Q= \begin{pmatrix}
2 & 3 & 5 \\ 1 & 2 & 4
\end{pmatrix} $$
corresponding to a $0$-dimensional subscheme $Z$
of degree $22$, whose ideal sheaf $\ii$ has resolution:
$$ 0\to \oo(-9)\oplus\oo(-8) \to \oo(-7)\oplus \oo(-6)\oplus \oo(-4) \to \ii\to 0.$$
Can one find a similar scheme in a general curve of degree $4$?
The answer is positive, as $4=a_n<b_{n-1}$, in the notation of 
Corollary \ref{cor of main}. 
Alternatively, notice that, adding the row $(-3\ -2 \ 0)$
(whose three entries are $4-7$, $4-6$, $4-4$)
and reordering, one ends up with the matrix:
 $$ \begin{pmatrix}
2 & 3 & 5 \\ 1 & 2 & 4 \\ -3 & -2 & 0
\end{pmatrix} $$
which satisfies the assumptions of the main theorem.
Notice that $m_{32}=-2<0$, but the submatrix obtained
by erasing the first $2$ rows and columns, has degree $0$.

Can one find a similar scheme in a general curve of degree $5$?
The answer is negative. Namely, $5<\min\{a_{n-1},b_{n-1}\}$
in the notation of Corollary \ref{cor of main}.
Observe that adding to $Q$ the row $(-2\ -1 \ 1)$
and reordering, one ends up with the matrix:
 $$ \begin{pmatrix}
2 & 3 & 5 \\ 1 & 2 & 4 \\ -2 & -1 & 1
\end{pmatrix} $$
Here $m_{32}<0$, but erasing the first $2$ rows and columns, the remaining
matrix has degree $1\neq 0,d$.

Notice that any quintic curve containing $Z$ 
corresponds to the product of the quartic generator 
and a linear form, hence cannot be irreducible
(of course, this is perfectly consistent with our Theorem \ref{main}).

Similarly, one shows that a general curve of degree $6$ or $7$
contains a subscheme with  dHB matrix equal to $Q$.
Arguing as above, since the Hilbert function of $Z$, at level $7$, is 
equal to the degree 
of $Z$, then one can find a subscheme with  dHB matrix equal to $Q$,
on a general curve of any degree $d\geq 7$.
\end{ex}

\begin{ex}
The previous example also points out a curious consequence.

Consider  the Hilbert scheme of
subsets of degree $22$ in $\pp^2$, and
let $T$ be the subvariety of subsets having a dHB equal to $Q$.
Consider the incidence variety 
$$\II(4) =\{(C,Z): C \mbox{ is a curve of degree $4$ containing }Z\in T \}$$
introduced in section \ref{inc}. 
Using  \cite{strat} Theorem 2.6, one computes that $\dim\II(4)=\dim(T)= 21$,
since every scheme like $Z$ sits in a unique quartic.
The main theorem implies that the natural projection $\II(4)\to \pp(H^0\oo(4))$
is dominant, with general fibers of dimension $7$.

Consider now the incidence variety 
$$\II(5) =\{(C,Z): C \mbox{ is a curve of degree $5$ containing }Z\in T \}$$
introduced in section \ref{inc}. 
One has $h^0(\ii(5))=3$, so $\dim\II(5)=\dim(T)+2= 23$,
which is bigger than $\dim(\pp(H^0\oo(5)))=20$. On the other hand, as we saw in
the example, the projection $\II(5)\to \pp(H^0\oo(5))$ is not dominant.
Indeed, the fibers of these projection have dimension at least $7$.
Notice that the image coincides with the space of quintics splitting 
in a quartic plus a line. This space has dimension $16$, which is exacly $23-7$.

As a consequence, we see that the natural projections
$\II(d)\to \pp(H^0\oo(d))$, introduced in section \ref{inc},
are non-necessarily of maximal rank.
\end{ex}

\begin{ex}
Let $I_1$ be the ideal of $Z_1 = $ 7 general points in $\pp^2$ 
and let $I_2$ be the ideal of  $Z_2 = $13 points on a conic, $Q$. 
 Let $I_Z = I_1 \cap I_2$ be the ideal of $Z = Z_1 \cup Z_2$.  
 Notice that any curve of degree $\leq 6$ necessarily has $Q$ 
 as a component, hence the general curve of degree $6$ 
 does not contain a set of points like $Z$.  

The difference of the Hilbert function of $R/I_Z$ is $\tt (1, 2, 3, 4, 5, 3, 2)$.  
The minimal free resolution of $I_Z$ is
$$
0 \rightarrow \mathcal O(-6) \oplus \mathcal O(-7) \oplus 
\mathcal O(-8)^2 \rightarrow \mathcal O(-5)^3 \oplus 
\mathcal O(-7)^2 \rightarrow \mathcal I_Z \rightarrow 0.
$$
Thus the dHB matrix for $R/I_Z$ is
$$ Q = \begin{pmatrix} 1 & 1 & 3 & 3 & 3 \\
1 & 1 & 3 & 3 & 3 \\
0 & 0 & 2 & 2 & 2 \\
-1 & -1 & 1 & 1 & 1 
\end{pmatrix}$$

Let us follow the procedure of the solution to Problem 5.1, 
asking if a general plane curve of degree 6 contains 
a zero-dimensional scheme with this dHB matrix.  
Since $d=6$, we add the row $(-1,-1,1,1,1)$:
$$  \begin{pmatrix}
1 & 1 & 3 & 3 & 3 \\
1 & 1 & 3 & 3 & 3 \\
0 & 0 & 2 & 2 & 2 \\
-1 & -1 & 1 & 1 & 1 \\
-1 & -1 & 1 & 1 & 1 
\end{pmatrix}$$

This matrix satisfies both conditions in Theorem \ref{main}, 
hence the theorem asserts that a general plane curve of degree $6$
 contains a zero-dimensional subset with a dHB matrix equal to $Q$.  

 However, this does not contradict the observation at 
 the beginning of this example.  Indeed, in the family 
 of zero-dimensional schemes in $\pp^2$ whose difference of
 the Hilbert function is $\tt (1, 2, 3, 4, 5, 3, 2)$ (an irreducible family), 
 the general element, $Y$, has minimal free resolution
$$
0 \rightarrow \mathcal O(-6) \oplus  \mathcal O(-8)^2 \rightarrow 
\mathcal O(-5)^3 \oplus \mathcal O(-7) \rightarrow \mathcal I_Y \rightarrow 0.
$$
The algorithm shows that a general curve of degree $6$ 
contains a set of points with this minimal free resolution, 
so adding a trivial summand $\mathcal O(-7)$ to both free modules gives $Q$.
\end{ex}

Corollary \ref{cor of main} can also be used as a way  to determine
quickly the existence of 
linear series on a general plane curve, with preassigned 
 index of speciality for  sums $D+zH$, $z\in\ZZ$.
Let us give one example.

\begin{ex}
Let $C$ be a general plane curve of degree 8.  
It is easy to compute, using the Brill-Noether theory,
that $C$ has (special) linear series $g^2_{20}$. 
For a divisor $D$ of degree 20, let us furthermore consider 
the following properties ($H$ is a linear divisor):

\begin{itemize}
\item[(A)] $D+H$ is non-special;
\item[(B)] $D-H$ is effective.
\end{itemize}

\noindent We will show that $C$ contains different 
{\em complete} $g_{20}^2$'s, whose divisors $D$ satisfy 
all possible combinations 
of properties (A) and (B).

Let $D$ be a divisor on $C$, which we will also consider as a 
subscheme of $\pp^2$.  Call $I_D$ its homogeneous ideal and $Hf$ its 
Hilbert function.

Then by Riemann-Roch and adjunction, 
$D$ belongs to a complete $g_{20}^2$ on $C$ if and only if $\dim (I_D)_{5} = 3$ 
(i.e. $Hf(5) = 18$).  Furthermore, $D+H$ is non-special 
if and only if $\dim (I_D)_4 = 0$ (i.e. $Hf(4) = 15$), and $D-H$ is effective 
if and only if $\dim (I_D)_6 > 8$ (i.e. $Hf(6) \leq 19$).  

In order for there to be a complete $g_{20}^2$,  
the possible Hilbert functions for $D$ are
\smallskip
$$\begin{matrix} 
 (a) & {\tt 1, 3, 6, 10, 15, 18, 20, 20, \dots} \\
 (b) & {\tt 1, 3, 6, 10, 15, 18, 19, 20, \dots} \\
 (c) & {\tt 1, 3, 6, 10, 14, 18, 20, 20, \dots} \\
 (d) & {\tt 1, 3, 6, 10, 14, 18, 19, 20, \dots} 
 \end{matrix}
$$ \smallskip
 
 Then one checks with our methods 
  that all four kinds of $g_{20}^2$'s exist on $C$, 
  and one sees immediately from the above that (a) has 
  property (A) only, (b) has properties (A) and (B), 
  (c) has neither property, and (d) has property (B) only.
\end{ex}

\end{document}